\documentclass{amsart}

\newtheorem{lemma}{Lemma}[section]
\newtheorem{theorem}[lemma]{Theorem}
\newtheorem{remark}[lemma]{Remark}

\newtheorem{coro}[lemma]{Corollary}
\newtheorem{definition}[lemma]{Definition}
\newtheorem{example}[lemma]{Example}

\title{ Levitan Almost Periodic Solutions of Linear Differential Equations}

\author{David ~Cheban}
\address[D. Cheban]{%
State University of Moldova\\ Faculty of Mathematics and
Informatics\\  Department of Mathematics\\ A. Mateevich Street 60\\
MD--2009 Chi\c{s}in\u{a}u, Moldova} \email[D.
Cheban]{cheban@usm.md, davidcheban@yahoo.com}

\date{\today}

\begin{document}
\begin{abstract}
{ The known Levitan's Theorem states that the linear differential
equation
$$
x'=A(t)x+f(t) \ \ \ (*)
$$
with Bohr almost periodic coefficients
$A(t)$ and $f(t)$ admits at least one Levitan almost periodic
solution if it has a bounded solution. The main assumption in this
theorem is the separation among bounded solutions of homogeneous
equations
$$
x'=A(t)x\ .\ \ \ (**)
$$
In this paper we prove that linear differential equation (*) with
Levitan almost periodic coefficients has a Levitan almost periodic
solution, if it has at least one bounded solution. In this case,
the separation from zero of bounded solutions of equation (**) is
not assumed. The analogue of this result for difference equations
also is given.

We study the problem of existence of Bohr/Levitan almost periodic
solutions for equation (*) in the framework of general
nonautonomous dynamical systems (cocycles).}

\end{abstract}

\maketitle

\section{Introduction}\label{Sec1}

This paper is dedicated to studying the problem of Levitan almost
periodicity of solutions for linear differential equation
\begin{equation}\label{eq01}
x'(t)=A(t)x(t)+f(t)
\end{equation}
with Levitan almost periodic in time coefficients $A(t)$ and
$f(t)$. We prove that if the functions $f$ and matrix-function
$A(t)$ are Levitan almost periodic and equation (\ref{eq01}) has a
bounded on real axis $\mathbb R$ solution, then equation
$(\ref{eq01})$ admits at least one Levitan almost periodic
solution.

Let $(X,\rho)$ be a complete metric space. Denote by $C(\mathbb
R,X)$ the space of all continuous functions $\varphi :\mathbb R
\to X$ equipped with the distance
$$
d(\varphi,\psi):=\sup\limits_{L>0}\min\{\max\limits_{|t|\le
L}\rho(\varphi(t),\psi(t)),L^{-1}\}.
$$
The space $(C(\mathbb R,X),d)$ is a complete metric space.

Let $h\in \mathbb R$ and $\varphi \in C(\mathbb R,X)$. Denote by
$\varphi^{h}$ the $h$-translation of function $\varphi$, i.e.,
$\varphi^{h}(t):=\varphi(t+h)$ for any $t\in\mathbb R$ and by
$\mathfrak N_{\varphi}:=\{\{h_k\}:\ \varphi^{h_k}\to \varphi\}$.
Note that the convergence $\varphi^{h_k}\to \varphi$ as $k\to
\infty$ means the convergence uniform on every compact
$[-l,l]\subset \mathbb R$ ($l>0$).

\begin{definition} \rm
Let $\varepsilon >0$. A number $\tau \in \mathbb R$ is called {\em
$\varepsilon$-almost period} of the function $\varphi$ if
$$
\rho(\varphi(t+\tau),\varphi(t))<\varepsilon
$$
or all $t\in\mathbb R$. Denote by $\mathcal
T(\varphi,\varepsilon)$ the set of $\varepsilon$-almost periods of
$\varphi$.
\end{definition}

\begin{definition} \rm (H. Bohr, 1923)
A function $\varphi \in C(\mathbb R,X)$ is said to be {\em Bohr
almost periodic} if the set of $\varepsilon$-almost periods of
$\varphi$ is {\em relatively dense} for each $\varepsilon >0$,
i.e., for each $\varepsilon >0$ there exists $l=l(\varepsilon)>0$
such that $\mathcal T(\varphi,\varepsilon)\cap
[a,a+l]\not=\emptyset$ for all $a\in\mathbb R$.
\end{definition}

\begin{definition} \rm
Let $\varphi \in C(\mathbb R,X)$ and $\psi \in C(\mathbb R,Y)$. A
function $\varphi\in C(\mathbb R,X)$ is called {\em Levitan almost
periodic} ($N$-almost periodic) if there exists a Bohr almost
periodic function $\psi \in C(\mathbb R,Y)$ such that $\mathfrak
N_{\psi}\subseteq \mathfrak N_{\varphi}$, where $Y$ is some metric
space (generally speaking $Y\not= X$).
\end{definition}

\begin{remark} 1. Every Bohr almost periodic function is Levitan almost
periodic.

2. The function $\varphi \in C(\mathbb R,\mathbb R)$ defined by
equality
$$
\varphi(t)=\dfrac{1}{2+\cos t +\cos \sqrt{2}t}
$$
is Levitan almost periodic, but it is not Bohr almost periodic
(because it is not bounded).
\end{remark}

In 1939 B. M. Levitan \cite{Lev_1939} published his paper, where
he studied the problem of existence of Levitan almost periodic
solutions of equation
\begin{equation}\label{eq0.1}
x'=A(t)x+f(t)\ \ (x\in\mathbb R^{n})
\end{equation}
with the matrix $A(t)$ and vector-function $f(t)$ Levitan almost
periodic.

Along with equation (\ref{eq0.1}), consider the homogeneous
equation
\begin{equation}\label{eq0.2}
x'=A(t)x\ .
\end{equation}

\begin{theorem}\label{tF} (Levitan's theorem
\cite{Lev_1939}-\cite{Lev-Zhi}) Linear differential equation
(\ref{eq0.1}) with Bohr almost periodic coefficients admits at
least one Levitan almost periodic solution, if it has a bounded
solution and each bounded on $\mathbb R$ solution $\varphi(t)$ of
equation (\ref{eq0.2}) is separated from zero, i.e.
\begin{equation}\label{eq0.4}
\inf\limits_{t\in \mathbb R}|\varphi(t)|>0.\nonumber
\end{equation}
\end{theorem}

Denote by $$H(A,f):=\overline{\{(A^{h},f^{h})|\ h\in R\}},$$ where
by bar is denoted the closure in the space $C(\mathbb R,[\mathbb
R^n])\times C(R,\mathbb R^n),$ where $[\mathbb R^n]$ the space of
all linear operators acting on the space $\mathbb R^n$.

\begin{theorem}\label{thZ} (Zhikov's theorem \cite{Zhi_1971})
Linear differential equation (\ref{eq0.1}) with Bohr almost
periodic coefficients admits at least one Levitan almost periodic
"limiting" solution, if it has a bounded solution, i.e., there
exists a limiting equation
\begin{equation}\label{eqZ1}
x'(t)=B(t)x(t)+g(t),
\end{equation}
where $(B,g)\in H(A,f)$.
\end{theorem}

Denote by $\Omega :=\{(B,g)\in H(A,f)|$ such that equation
$(\ref{eqZ1})$ has a Levitan almost periodic solution $\}$. Zhikov
proved that the set $\Omega$ has a second category of Baire.

\textbf{Open problem }(V. V. Zhikov \cite{Zhi_1971}). Is equality
$\Omega =H(A,f)$ true? In other words. Can we state that every
equation (\ref{eq0.1}) admits at least one Levitan almost periodic
solution if (\ref{eq0.1}) has a bounded on $\mathbb R$ solution?

From our main result it follows the positive answer to this
question. Namely, the following statement holds.

\begin{theorem} Suppose that the following conditions are
fulfilled:
\begin{enumerate}
\item the matrix $A(t)$ and function $f(t)$ are Levitan almost
periodic; \item equation
\begin{equation}\label{eq1}
x'(t)=A(t)x(t)+f(t) \ \  (x\in \mathbb R^{n})
\end{equation}
has a bounded on $\mathbb R$ solution.
\end{enumerate}
Then equation (\ref{eq1}) has at least one Levitan almost periodic
solution.
\end{theorem}

\begin{coro}\label{corI1} Under the conditions of Theorem if the
coefficients $A(t)$ and $f(t)$ are Bohr almost periodic, then
equation (\ref{eq1}) admits at least one Levitan almost periodic
solution.
\end{coro}

\begin{remark}\label{remI1} 1. Under the conditions of Corollary
\ref{corI1}, equation (\ref{eq1}) in general may not have Bohr
almost periodic solutions \cite{Joh,Zh-Le} (see also
\cite[ChVIII]{Lev-Zhi} and \cite{OT}).

2. Notice that Corollary \ref{corI1} before was established by B.
M. Levitan \cite{Lev_1939}-\cite{Lev_1953} with the additional
assumption that the bounded solutions of equation
\begin{equation}\label{eq2}
x'(t)=A(t)x(t) \nonumber
\end{equation}
are separated from zero.
\end{remark}

This paper is organized as follow.

In Section \ref{Sec2} we collect some notions and facts from the
theory of dynamical systems (some classes of Poisson stable
motions, comparability by character of recurrence of Poisson
stable motions, nonautonomous dynamical systems, conditionally
compactness).

Section \ref{Sec3} is dedicated to the study the problem of
existence of a common fixed point for noncommutative affine
semigroup of mappings (Theorems \ref{th12.3.8} and
\ref{th12.3.8*}).

In section \ref{Sec4} we study the problem of existence at least
one compatible motion/solution of linear nonhomogeneous dynamical
system (Theorem \ref{thB2}).

Section \ref{Sec5} is dedicated to the application of our general
results obtained in Section \ref{Sec4} to the linear differential
(Subsection \ref{Ssec5.1}) and difference (Subsections
\ref{Ssec5.2} and \ref{Ssec5.3}) equations.

\section{Some motions and facts from the theory of dynamical
systems}\label{Sec2}

\subsection{Poisson stable motions}

Let $(X,\rho)$ be a complete metric space with metric $\rho$,
$\mathbb{R}$ $(\mathbb{Z})$ be a group of real (integer) numbers,
$\mathbb{R_{+}}$ $(\mathbb{Z_{+}})$ be a semigroup of the
nonnegative real (integer) numbers, $\mathbb S$ be one of the two
sets $\mathbb{R}$ or $\mathbb{Z}$, $\mathbb{T}\subseteq\mathbb{S}$
$(\mathbb{S_{+}}\subseteq\mathbb{T}$) be a sub-semigroup of
additive group $\mathbb{S}$ and $\mathbb T_{+}:=\{t\in \mathbb T|\
t\ge 0\}$.

Let $(X,\mathbb T,\pi)$ be a dynamical system. Let us recall the
classes of Poisson stable motions we study in this paper, see
\cite{Che_2020,Sell71,Sch72,scher85,sib} for details.

\begin{definition}\label{def-stat}\rm
A point $x \in X $ is called {\em stationary} (respectively, {\em
$\tau$-periodic}) if $\pi(t,x)=x$ (respectively,
$\pi(t+\tau,x)=\pi(t,x)$) for all $t\in\mathbb T$.
\end{definition}

\begin{definition}\label{def4.2}\rm
For given $\varepsilon>0$, a number $\tau \in \mathbb T$ is called
a {\em $\varepsilon$-shift of $x$} (respectively, {\em
$\varepsilon$-almost period of $x$}), if $\rho
(\pi(\tau,x),x)<\varepsilon$ (respectively, $\rho (\pi(\tau
+t,x),\pi(t,x))<\varepsilon$ for all $t\in \mathbb T$).
\end{definition}

\begin{definition}\label{def14.3}\rm
A point $x \in X $ is called {\em almost recurrent} (respectively,
{\em Bohr almost periodic}), if for any $\varepsilon >0$ there
exists a positive number $l$ such that any segment of length $l$
contains a $\varepsilon$-shift (respectively, $\varepsilon$-almost
period) of $x$.
\end{definition}

\begin{definition}\label{def4.4}\rm
If a point $x\in X$ is almost recurrent and its trajectory
$\Sigma_{x}:=\{\pi(t,x):\ t\in\mathbb T\}$ is precompact, then $x$
is called {\em (Birkhoff) recurrent}.
\end{definition}

\begin{definition}\rm
A point $x\in X$ is called {\em Levitan almost periodic}
\cite{Lev-Zhi} (see also \cite{Bro79,Che_2008,Che_2020,Lev_1953}),
if there exists a dynamical system $(Y,\mathbb T,\sigma)$ and a
Bohr almost periodic point $y\in Y$ such that $\mathfrak
N_{y}\subseteq \mathfrak N_{x}.$
\end{definition}


\begin{definition} \rm
A point $x\in X$ is called {\em almost automorphic} if
\begin{enumerate}
\item it is stable in the sense of Lagrange, i.e., its trajectory
$\Sigma_{x}:=\{\pi(t,x)|\ t\in \mathbb T\}$ is relatively compact
and \item $x$ is Levitan almost periodic.
\end{enumerate}
\end{definition}




\subsection{Comparability of motions by the character of
recurrence}

Following B. A. Shcherbakov \cite{Sch75,scher85} (see also
\cite{Che_1977}, \cite[ChI]{Che_2009}) we introduce the notion of
comparability of motions of dynamical system by the character of
their recurrence. While studying stable in the sense of Poisson
motions this notion plays the very important role (see, for
example, \cite{Sch72,scher85}).

Let $(X,\mathbb T,\pi)$ and $(Y,\mathbb T,\sigma)$ be dynamical
systems, $x\in X$ and $y\in Y$. Denote by $\mathfrak
M_{x}:=\{\{t_n\}:$ such that $\{\pi(t_n,x)\}$ converges as $n\to
\infty \}$, $\mathfrak N_{x}:=\{\{t_n\}:$ such that $\pi(t_n,x)\to
x$ as $n\to \infty \}$ and  $\mathfrak
N_{x}^{+\infty}:=\{\{t_n\}\in \mathfrak N_{x}:$ such that $t_n\to
+\infty$ as $n\to \infty \}$.

\begin{definition}
A point $x_0\in X$ is called comparable by the character of
recurrence with $y_0\in Y$ if there exists a continuous mapping
$h:\Sigma_{y_0}\mapsto \Sigma_{x_0}$ satisfying the condition
\begin{equation}\label{eqC1_1}
h(\sigma(t,y_0))=\pi(t,x_0)\ \ \mbox{for any}\ t\in \mathbb R
.\nonumber
\end{equation}
\end{definition}

\begin{definition}\label{defNDS} Let $(X,\mathbb T_1,\pi)$ and $(Y,\mathbb
T_2,\sigma)$ be two dynamical systems, $\mathbb T_1\subseteq
\mathbb T_2$ and $h:X\mapsto Y$ be a homomorphism of $(X,\mathbb
T_1,\pi)$ on $(Y,\mathbb T_2,\sigma)$. A triplet $\langle
(X,\mathbb T_1,\pi), (Y,\mathbb T_2,\sigma), h \rangle$ is said to
be a nonautonomous dynamical system \cite[ChI]{Che_2020}.
\end{definition}

\begin{theorem}\label{thPC2}\cite{Che_2019} Let $\langle (X,\mathbb T_1,\pi), (Y,\mathbb T_2,\sigma),h
\rangle$ be a nonautonomous dynamical system and $x_0\in X$ be a
conditionally Lagrange stable point (i.e., the set $\Sigma_{x_0}$
is conditionally precompact), then if $H(x_0)\bigcap X_{y_0}$
consists a single point $\{x_0\}$, where $y_0:=h(x_0)$, then
$\mathfrak N_{y_0}\subseteq \mathfrak N_{x_0}$.
\end{theorem}

\begin{definition}\label{def9.2.1}
The point $y\in Y $ is called (see, for example, \cite{scher85}
and \cite{sib}) positively stable in the sense of Poisson if there
exists a sequence $t_{n} \to +\infty $ such that $\sigma^{t_{n}}y
\to y.$
\end{definition}

\begin{theorem}\label{thPC3}\cite{CC_2009,Sch75} Let $y\in Y$
be Poisson stable in the positive direction, then the following
statement are equivalent:
\begin{enumerate}
\item the point $x\in X$ is comparable with $y\in Y$ by the
character of recurrence; \item $\mathfrak N_{y}^{+\infty}\subseteq
\mathfrak N_{y}^{+\infty}$; \item for any $\varepsilon >0$ there
exists a $\delta =\delta(\varepsilon)>0$ such that
$d(\sigma(\tau,y),y)<\delta$ implies
$\rho(\pi(\tau,x),x)<\varepsilon$, where $d$ (respectively,
$\rho$) is the distance on the space $Y$ (respectively, on the
space $X$).
\end{enumerate}
\end{theorem}

\begin{theorem}\label{thPC4}\cite{Sch75} Suppose that the point
$x$ is comparable with $y\in Y$ by the character of recurrence. If
the point $y$ is stationary (respectively, $\tau$-periodic,
Levitan almost periodic, almost recurrent in the sense of Bebutov,
Poisson stable), the point $x$ is so.
\end{theorem}

\subsection{Some general facts about nonautonomous dynamical
systems}

\begin{definition}\label{def9.2.2}
(Conditional compactness). Let $(X,h,Y)$ be a fibre space, i.e.,
$X$ and $Ya$ be two metric spaces and $ h: X \to Y$ be a
homomorphism from $X$ into $Y$. The subset $ M \subseteq X$ is
said to be conditionally relatively compact\index{conditionally
relatively compact}, if the pre-image $h^{-1}(Y ')\bigcap M $ of
every relatively compact subset $ Y'\subseteq Y $ is a relatively
compact subset of $X$, in particularly $M_{y}:=h^{-1}(y)\bigcap M$
is relatively compact for every $y $. The set $ M $ is called
conditionally compact if it is closed and conditionally relatively
compact.
\end{definition}

\begin{example}\label{ex9.2.3}
Let $K$ be a compact space, $X:=K\times Y $, $h= pr _{2} : X \to
\Omega,$ then the triplet $(X,h,Y)$ be a fibre space, the space
$X$ is conditionally compact, but not compact.
\end{example}

Let $ \langle (X,\mathbb T _{+},\pi),(Y,\mathbb T,\sigma),h\rangle
$ be a nonautonomous dynamical system and $y\in Y$ be a positively
Poisson stable point. Denote by $$ \mathcal{E}_{y}^{+}:=\{\xi
|\quad \exists \{t_{n}\}\in \mathfrak N _{y}^{+\infty} \quad
\mbox{such that}\quad \pi ^{t_{n}}|_{X_{y}}\to \xi\},$$ where
$X_{y}:=\{x\in X|\quad h(x)=y\}$ and $\to$ means the pointwise
convergence.

\begin{lemma}\label{l9.2.4}\cite{Che_2020}
Let $y\in Y$ be a positively Poisson stable point, $\langle
(X,\mathbb T_{+},\pi ),$ $(Y,\mathbb T,\sigma),h\rangle $ be a
nonautonomous dynamical system and $X$ be a conditionally compact
space, then $\mathcal{E}^{+}_{y} $ is a nonempty compact
subsemigroup of the semigroup $X_{y}^{X_{y}}$ (w.r.t. composition
of mappings).
\end{lemma}
\begin{proof} Let $\{t_{n}\}\in \mathfrak N ^{+\infty}_{y}$, then $\sigma _{t_{n}}y \to y$
and, consequently, the set
$$
Q:= \overline {\bigcup \{\pi ^{t_{n}}(X_{y})| n \in \mathbb N \}}
$$
is compact, because $X$ is conditionally compact. Thus $\{\pi
^{t_{n}}|_{X_{y}}\} \subseteq Q^{X_{y}}$ and according to
Tykhonov's theorem this sequence is precompact. Let $\xi$ be a
limiting point of $\{\pi ^{t_{n}}|_{X_{y}}\}$, then $\xi \in
\mathcal{E}_{y}^{+}$ and, consequently, $\mathcal{E}^{+}_{y}\not=
\emptyset .$

We note that $\mathcal{E}^{+}_{y} \subseteq X_{y}^{X_{y}}$ and,
consequently, $\mathcal{E}^{+}_{y}$ is precompact. Let now $\xi
_{1},\xi _{2} \in \mathcal{E}^{+}_{y}$, we will prove that $\xi
_{1}\cdot\xi _{2} \in \mathcal{E}^{+}_{y}$. Since $\xi _{1},\xi
_{2} \in \mathcal{E}^{+}_{y}$, then there are two sequences
$\{t_{n}^{i}\} \in \mathfrak N ^{+\infty}_{y}\ (i=1,2)$ such that
$$ \pi ^{t_{n}^{i}}|_{X_{y}} \to \xi _{i}\quad (i=1,2).$$
Denote by $\xi :=\xi _{1}\cdot \xi _{2} \in
X_{\omega}^{X_{y}}\subseteq
 \tilde{Q}^{X_{y}}$, where
$$
\tilde{Q}:= \overline {\bigcup \{\pi ^{t^{i}_{n}}(X_{y})|\ i=1,2;\
n \in \mathbb N \}} .
$$
Then we have
\begin{equation}\label{eqCC1}
\pi ^{t_{n}^{1}} \cdot \xi _{2} \to \xi _{1}\cdot \xi _{2}=\xi
\end{equation}
as $n \to + \infty.$ Let $U_{\xi} \subset \tilde{Q}^{X_{y}}$ be an
arbitrary open neighborhood of point $\xi$ in $
\tilde{Q}^{X_{y}}$, then from relation (\ref{eqCC1}) results that
there exists a number $n_{1}(\xi)\in \mathbb N $ such that $ \pi
^{t_{n}^{1}} \cdot \xi _{2} \in U_{\xi} $ for any $n\ge
n_{1}(\xi).$ Now we fix $n\ge n_{1}(\xi),$ then there exist an
open neighborhood $ U_{\pi ^{t_{n}^{1}}\cdot \xi _{2}} \subset
U_{\xi}$ of point $\pi ^{t_{n}^{1}}\cdot \xi _{2} \in
\tilde{Q}^{X_{y}}$ and a number $m_{n}\in \mathbb N $ such that
$$
\pi ^{t^{1}_{n}} \cdot \pi ^{t_{m}^{2}} |_{X_{y} } \in U_{\pi
^{t_{n}^{1}} \cdot \xi _{2}}
$$
for any $n\ge n_{1}(\xi)$ and $m\ge m_{n}(\xi) $ and,
consequently,
$$
\pi ^{t^{1}_{n}}\cdot \pi ^{t_{m}^{2}}|_{X_{y}}\in U_\xi
$$
for any $n \ge n_{1}(\xi)$ and $m \ge m_{n}(\xi).$ Thus from
sequence $\{\pi ^{t_n^1+t_m^2}|_{X_{y}}\}$ it is possible to
extract a subsequence $\{\pi ^{t_{n_k}^1+t_{m_k}^2}|_{X_{y}}\}
\quad (t_{n_k}^1+t_{m_k}^2 \to + \infty )$ such that $ \pi
^{t_{n_k}^1+t_{m_k}^2}|_{X_{y}}\to \xi $ and, consequently, $ \xi
=\xi _{1}\cdot \xi _{2}\in \mathcal{E}^{+}_{y}.$

To finish the proof of Lemma it is sufficient to show that
$\mathcal{E}_{y}^{+}$ is closed in $X_{y}^{X_{y}}$. Let
$\{\varepsilon_k\}\to 0$ be a monotone decreasing positive
sequence, $\{\xi_k\}\subset \mathcal{E}_{y}^{+}$ such that
$\xi_k\to \xi$ in $X_{y}^{X_{y}}$ and $\{t_n^{(k)}\}\in \mathfrak
N_{y}^{+\infty}$ such that $\{\pi^{t_{n}^{(k)}}\}\to \xi_{k}$ as
$n\to \infty$ for every $k\in \mathbb N$. Then for each
$k\in\mathbb N$ there exists $n_k\in\mathbb N$ such that
\begin{equation}\label{eqN1}
\rho(\sigma(t_{n}^{(k)}, y),y)<\varepsilon_{k}\ \mbox{for all}\
n\ge n_{k}\ \ \mbox{and}\ \ t^{(k)}_{n_k}>k.
\end{equation}
We will establish that the set
$$
Q_{1}:=\overline{\bigcup \{\pi^{t_n^{(k)}}(X_{y}):\ n\ge n_k,\ \
k\in \mathbb N\}}
$$
is compact. If
$$
\{x^{l}\}\subseteq \bigcup \{\pi^{t_n^{(k)}}(X_{y}):\ n\ge n_k,\ \
k\in \mathbb N\},
$$
then for every $l\in\mathbb N$ there exist $m_{l}>n_{l}$ and
$\bar{x}_{l}\in X_{\omega}$ such that
$x_{l}=\pi(t^{(l)}_{m_l},\bar{x}_l)$. From (\ref{eqN1}) it follows
that $\{t_{m_l}^{(l)}\}\in\mathfrak N_{y}^{+\infty}$. Since the
set $X$ is conditionally compact, then from the sequence $\{x_l\}$
we can extract a convergent subsequence.

If $\xi \in \mathcal{E}_{y}^{+}\subseteq X_{y}^{X_{y}}\subseteq
Q_2^{X_{y}}$, where $Q_2:=X_{y}\bigcup Q_1$, and $U(\xi)$ is an
arbitrary neighborhood of $\xi$ in $Q_2^{X_{y}}$, then there
exists a number $l_0=l(\xi)\in\mathbb N$ such that $\xi_l \in
U(\xi)$ for all $l\ge l_0$. Denote by $V_l$ a neighborhood of
$\xi_l$ such that $V_l\subset U(\xi)$. Since $\xi_l\in
\mathcal{E}_{y}^{+}$, then there exists a number $m_l>n_l$ such
that $\pi^{t_l}\in V_l$, where $t_l:=t^{(l)}_{m_l}$.

From (\ref{eqN1}) we obtain $\{t_l\}\in\mathfrak N_{y}^{+\infty}$
and $\pi^{t_l}\big{|}_{X_{y}}\to \xi$ as $l\to \infty$,  in
$Q_2^{X_{y}}$. Lemma is proved.
\end{proof}

\section{Fixed point theorem for noncommutative affine semigroup of
mappings}\label{Sec3}

Let $(X,\rho)$ be a metric space. Denote by $B[a,r]:=\{x\in X|\
\rho(x,a)\le r\}$, where $a\in X$ and $r\ge 0$. For any $
x_{1},x_{2} \in X$ and $\alpha \in [0,1]$ denote by
$S(\alpha,x_1,x_2)$ the intersection of $B[x_{1},\alpha r]$ and
$B[x_{2},(1-\alpha)r]$, where $r=\rho (x_{1},x_{2}))$.

\begin{definition}\label{def12.22}
A metric space $ (X,\rho)$ is called:
\begin{enumerate}
\item a metric space with a convex structure \cite{Tak_1970} if
there exists a mapping $W:[0,1]\times X\times X\to X$ satisfying
\begin{equation}\label{eqCS1}
\rho(u,W(\alpha,x_1,x_2))\le \alpha \rho(u,x_1)\nonumber
+(1-\alpha)\rho(u,x_2) ;
\end{equation}
\item strictly convex \cite{Tak_1970} if for any $x,y\in X$ and
$\alpha \in [0,1]$ there exists a unique element $x $
($x=S(\alpha,x_1,x_2)$) such that $\rho(x,x_1)=\alpha
\rho(x_1,x_2)$ and $\rho(x,x_2)=(1-\alpha)\rho(x_1,x_2)$; \item
strongly convex \cite{Bul_1999,Hit_1982} (or strictly convex space
with convex round balls), if $(X,\rho)$ is a strictly convex
metric space and for any $x_1,x_2,x_3\in X$ ($x_2\not= x_3$) and
$\alpha \in (0,1)$ there holds the inequality $ \rho (x_{1
},S(\alpha,x_{2},x_{3}))<\max \{\rho (x_{1},x_{2}), \rho
(x_{1},x_{3})\}$.
\end{enumerate}
\end{definition}

\begin{definition}\label{def12.23}
Let $X$ be a metric space with a convex structure (respectively,
strictly convex or strongly convex). A subset $M$ of $X$ is said
to be convex (respectively, strictly convex or strongly convex),
if $S(\alpha,x_{1},x_{2})\in M$ for any $ \alpha \in (0,1)$ and
$x_{1},x_{2} \in M $.
\end{definition}

\begin{remark}\label{RSC1} 1. Closed balls may be non convex sets
and intersection of convex sets may be non convex set
\cite{Gal_1992}.

2. Intersection of convex sets is convex set in strictly convex
metric space \cite{Gal_1992}.

3. There exist strictly convex metric spaces in which closed balls
are not convex \cite{BV_IJPAM_2005}.

4. The closed ball $B[c,r]$ for every $r>0$ and every $c\in X$ is
a convex set in the strongly convex metric space $(X,\rho)$
\cite{Bul_1999}
\end{remark}

\begin{definition}
A Banach space $X$ is said to be:
\begin{enumerate}
\item uniformly convex, if the inequality $|p_1-p_2|\ge \delta
\max\{|p_1|,|p_2|\}$ implies $|\frac{1}{2}(p_1+p_2)|\le
(1-\varphi(\delta))\max\{|p_1|,|p_2|\}$ ($\varphi(\delta)>0$ for
any $0<\delta \le 2$); \item strictly convex, if for any $x,y\in
X$ with $|x|=|y|=1$ and $x\not=y,$ $|\lambda x+(1-\lambda)y|<1$
for any $\lambda \in (0,1).$
\end{enumerate}
\end{definition}

\begin{remark}\label{remCS1}
1. Uniformly convex Banach spaces are strictly convex, but the
converse is not true.

2. If $(X,|\cdot|)$ is a strictly convex Banach space, then the
metric space $(X,\rho)$ ($\rho(x_1,x_2):=|x_1 -x_2|$) is strictly
convex (see, for example, \cite{Bul_1999,Gal_1992}).

3. If $M$ is a convex subset of strictly convex Banach space
$(X,|\cdot|)$, then the metric space $(M,\rho)$
($\rho(x_1,x_2):=|x_1 -x_2|$) is  strictly convex.

4. Every convex closed subset $X$ of the Hilbert space $H$
equipped with metric $\rho (x_{1},x_{2}) = \vert x_{1}-x_{2}\vert$
is a strongly convex metric space.
\end{remark}

\begin{lemma}\label{lUC1} If $(X,|\cdot|)$ is a uniformly convex Banach space, then the
metric space $(X,\rho)$ ($\rho(x_1,x_2):=|x_1 -x_2|$) is strongly
convex.
\end{lemma}
\begin{proof} By Remark \ref{remCS1} (items 1 and 2) the metric space
$(X,\rho)$ is strictly convex. Let now $x_1,x_2,x_3\in X$
($x_2\not= x_3$), $\alpha \in (0,1)$ and $x:=S(\alpha,x_2,x_3)$.
Denote by $p_1:=x_2-x_1$ and $p_2:=x_3-x_1$, then we have
$$
\rho(x_1,x)=|\frac{1}{2}(p_1+p_2)|\le
(1-\varphi(\delta))\max\{|p_1|,|p_2|\}=
$$
$$
(1-\varphi(\delta))\max\{\rho(x_1,x_2),\rho(x_1,x_3)\}<\max\{\rho(x_1,x_2),\rho(x_1,x_3)\}
$$
because $|p_1|=\rho(x_1,x_2)$ and $|p_2|=\rho(x_1,x_3)$. Lemma is
proved.
\end{proof}

For any subset $C$ of $X$ we denote by $co C$ (respectively,
$\overline{co} C$) the convex envelop (respectively, closed convex
envelope) of $C,$ i.e., $\overline{co} C$ (respectively,
$\overline{co} C$) is the intersection of all metric-covex
(respectively, closed, metric-convex) sets containing $C$.

\begin{definition}\label{defE1} A mapping $f:M\to M$ of compact
strictly metric-convex space $(M,\rho)$ is said to be:
\begin{enumerate}
\item segment preserving if $f([x_1,x_2])\subseteq
[f(x_1),f(x_2)]$, where $[x_1,x_2]:=$ $\{S(\alpha,$ $x_1,$ $x_2)|$
$\ 0\le \alpha\le 1\}$, for any $x_1,x_2\in M$; \item quasi-affine
\cite{Shu_1979} if $f(co A)\subseteq co f(A)$ for any subset $A$
of $M$; \item strongly quasi-affine if $f(\overline{co}
A)\subseteq \overline{co} f(A)$ for any subset $A$ of $M$; \item
affine if $f(S(\alpha,x_1,x_2))=S(\alpha;f(x_1),f(x_2))$ for any
$x_1,x_2\in M$ and $\alpha \in [0,1]$.
\end{enumerate}
\end{definition}

\begin{remark}\label{remE1} 1. If the mapping $f:M\to M$ is
quasi-affine, then it is segment preserving because
$[x_1,x_2]=co(\{x_1,x_2\})$.

2. If $M=[0,1]$ and $f:[0,1]\to [0,1]$ is a continuous and
strongly monotone, then it is quasi-affine \cite{Shu_1979}.

3.  If the mapping $f:M\to M$ is affine, then it is quasi-affine.
\end{remark}

\begin{lemma}\label{lQ1} Let $M$ be a nonempty compact convex
subset in locally convex vectorial space $V$. If the quasi-affine
mapping $f:M\to M$ is continuous, then it is strongly
quasi-affine.
\end{lemma}
\begin{proof} Let $A\subseteq M$, then we have $\overline{co}(f(A))\subseteq
M$. Since $f$ is quasi-affine we obtain
\begin{equation}\label{eqEE1}
f(co(A))\subseteq co(f(A)) .\nonumber
\end{equation}
On the other hand $f$ is a continuous map and, consequently,
\begin{equation}\label{eqEE2}
f(\overline{co}(A))\subseteq \overline{f(co(A))}\subseteq
\overline{co}f(A)\nonumber
\end{equation}
for any subset $A\subseteq M$. Lemma is proved.
\end{proof}

Let $ x_{0} $ be an arbitrary point of $M$. We denote by
\begin{equation}\label{eqE1.1*}
\ell (x)= \sup \limits_{\xi \in E} \rho (\xi (x), x_{0})
,\nonumber
\end{equation}
where $E$ is a compact sub-semigroup of the semigroup $M^{M}$.

\begin{lemma}\label{lE1.1} Let $(M,\rho)$ be a compact metric space,
$x_0\in M$ and $E$ be a compact sub-semigroup of the semigroup
$M^{M}$. Then for any $x\in M$ there exists at least one $\xi\in
E$ such that
\begin{equation}\label{eqE1_1*}
\ell(x)=\rho(\xi(x),x_0) .\nonumber
\end{equation}
\end{lemma}
\begin{proof} By definition of $l(x)$ there exists a sequence $\{\xi_{n}\}\subseteq
E$such that
\begin{equation}\label{eqE2*}
\ell(x):=\sup\limits_{n\to \infty}\rho(\xi_{n}(x),x_0).
\end{equation}
Since $E$ is a compact sub-semigroup of semigroup $M^{M}$ without
loss of generality we may suppose that $\{\xi_{n}\}$ is convergent
in $M^{M}$. This means that there exists $\xi\in E$ such that
$\xi_n\to \xi$ as $n\to \infty$. In particular we have
\begin{equation}\label{eqE3*}
\xi(p):=\sup\limits_{n\to \infty}\xi_{n}(p)
\end{equation}
for any $p\in M$. From (\ref{eqE2*}) and (\ref{eqE3*}) we obtain
\begin{equation}\label{eqE4*}
\ell(x):=\sup\limits_{n\to
\infty}\rho(\xi_{n}(x),x_0)=\rho(\sup\limits_{n\to
\infty}\xi_{n}(x),x_0)=\rho(\xi(x),x_0).\nonumber
\end{equation}
\end{proof}

\begin{lemma}\label{lE2}
Let $(M,\rho)$ be a compact metric space. Suppose that the
following conditions are fulfilled:
\begin{enumerate}
\item $E$ is a compact sub-semigroup of the semigroup $M^{M}$;
\item every $\xi\in E$ is continuous.
\end{enumerate}
Then there exists a point $x'\in M$ such that $l(x')=l_0$, where
\begin{equation}\label{eq5*}
\ell_0:=\inf\limits_{x\in M}\ell(x) .
\end{equation}
\end{lemma}
\begin{proof}
Let $\{x_{n}\} \subseteq M $ be a minimizing sequence for
(\ref{eq5*}), i.e.,
\begin{equation}\label{eq12.3.2}
\ell _{0} \le \ell (x_{n}) \le \ell _{0} + \frac{1}{n}\nonumber
\end{equation}
for all $ n \in \mathbb N $. Since the set $M$ is compact, we may
suppose that the sequence $\{x_{n}\} $ is convergent. Let $
x'=\lim \limits _{n \to +\infty} x_{n}$. We will show that $\ell
(x')\le \ell_{0}$. In fact. If we suppose that it is not true,
then there exist $\xi'\in E$ and $\varepsilon_{0}>0$ such that
\begin{equation}\label{eqE1*}
\ell(x')=\rho(\xi'(x'),x_0)>\ell_0+2\varepsilon_0 .
\end{equation}

Since $\ell(x_n)\to \ell_0$ as $n\to \infty$, then there exists a
number $n_0\in\mathbb N$ such that
\begin{equation}\label{eqE2.1*}
-\varepsilon_0 +\ell_0<\ell(x_n)<\ell_0+\varepsilon_0\nonumber
\end{equation}
for any $n\ge n_0$. On the other hand
\begin{equation}\label{eqE3}
\rho(\xi'(x_n),x_0)\le \sup\limits_{\xi\in
E}\rho(\xi(x_n),x_0)=\ell(x_n)<\ell_0+\varepsilon_0
\end{equation}
for any $n\ge n_0$. Passing to the limit in (\ref{eqE3}) as $n\to
\infty$  and taking in consideration the continuity of $\xi'$ we
obtain
\begin{equation}\label{eqE4}
\ell(x')=\lim\limits_{n\to \infty}\rho(\xi'(x_n),x_0)\le
\ell_0+\varepsilon_0 .
\end{equation}
Inequalities (\ref{eqE1*}) and (\ref{eqE4}) are contradictory. The
obtained contradiction proves our statement, i.e., $\ell(x')\le
\ell_0$. Since $ \ell (x) \ge \ell _{0} $ for all $ x \in M $ then
$ \ell (x')=\ell _{0}.$
\end{proof}

\begin{theorem}\label{th12.3.8}
Let $(M,\rho)$ be a compact strongly convex metric space. Suppose
that the following conditions are fulfilled:
\begin{enumerate}
\item  $E$ is a compact sub-semigroup of the semigroup $M^{M}$;
\item every $\xi\in E$ is continuous and quasi-affine.
\end{enumerate}
Then there exists a common fixed point $\bar{x}\in M $ of $ E $,
i.e., $ \xi (\bar{x})=\bar{x} $ for any $ \xi \in E$.
\end{theorem}
\begin{proof}
We put $ M' = \{ x \in M \ |\ \ell (x)= \ell _{0}\ \} $. By Lemma
\ref{lE2} we have $ M' \not= \emptyset $. The set $ M' $ is
invariant with respect to the semigroup $E$, i.e., $ \xi (M')
\subseteq M' $ for any $ \xi \in E $. In fact. Let $\eta \in E$, $
x' \in M'$ and $\tilde{\xi}\in E$, then we have $\ell
(\eta(x'))\ge \ell_{0}$. On the other hand taking in consideration
that $E$ is a semigroup then we obtain $\tilde{\xi}\eta\in E$ and
$$
\ell (\eta (x'))=\sup \limits _{\tilde{\xi} \in E} \rho
(\tilde{\xi} (\eta (x')), x_{0}) \le \sup \limits _{\xi \in E}
\rho (\xi (x'),x_{0}) = \ell (x')=\ell _{0}
$$
and, consequently, $ \ell (\eta (x')) =\ell _{0}$. We will show
now that the set $ M'$ consists of a single point. In fact, if we
suppose the contrary, then there exist $x_{1},x_{2} \in M' \
(x_{1}\not= x_{2}) $. We consider $ x= S(\frac{1}{2},x_{1},x_{2})
$, which is an element from $M$, because $M$ is a strongly convex
metric space. According to the quasi-affinity of $\xi$ we have
$\xi (co\{x_1,x_2\})\subseteq co\{\xi(x_1),\xi(x_2)\}\subseteq M$.
Since the semigroup $E$ is compact, then under the conditions of
Theorem \ref{th12.3.8} there exists $ \xi \in E $ such that $ \ell
_{0} \le \ell (x)=\rho (\xi (x),x_{0})$. On the other hand,
according to the strongly convexity of $M$ and quasi-affinity of
$\xi$ we have $ \rho (\xi (x),x_{0}) < \max \{\rho (\xi
(x_{1}),x_{0}), \rho (\xi (x_{2}),x_{0})\} = \ell _{0} $ and,
consequently, $ \ell (x) < \ell _{0}$. The obtained contradiction
proves that $ M'$ contains a unique point $\bar{x}$. Taking into
account the invariance of the set $ M'$ with respect to the
semigroup $E$, we have $ \xi (\bar{x})=\bar{x} $ for any $\xi \in
E $. Theorem is proved.
\end{proof}

Let $\mathfrak B$ be a finite-dimensional Banach space over the
field $\mathbb K$ ($\mathbb K =\mathbb R$ or $\mathbb C$). Denote
by $\mathbb K^{n}:=\mathbb K\times \mathbb K\times \ldots \times
K$, then by equality
\begin{equation}\label{eqSP1}
\langle \xi,\eta\rangle :=\sum_{i=1}^{n}\xi_{i}\bar{\eta_{i}}
\nonumber
\end{equation}
for any $\xi,\eta\in \mathbb K^{n}$, is defined a scalar product
on $\mathbb K^{n}$.

\begin{theorem}\label{thE1}\cite[ChI]{LS_1975} Let $(\mathfrak B,||\cdot||)$ be a finite-dimensional Banach space over the
field $\mathbb K$, then the following statements hold:
\begin{enumerate}
\item the space $(\mathbb K^{n},|\cdot|)$, where
$|\cdot|^2:=\langle \cdot,\cdot \rangle$, is an $n$-dimensional
Hilbert space; \item the Banach spaces $(\mathfrak B,||\cdot||)$
and $(\mathbb K^n,|\cdot|)$ are isomorphic.
\end{enumerate}
\end{theorem}

\begin{lemma}\label{lE1} Let $U$ be a linear isomorphism between
the Banach spaces $\mathfrak B$ and $\mathbb K^n$, $M$ be a
compact convex subset of $\mathfrak B$, $\tilde{M}:=U(M)$ and $E$
be a compact sub-semigroup of the semigroup $M^M$. Then the
following statements hold:
\begin{enumerate}
\item $\tilde{M}$ is compact and convex subset of $\mathbb K^{n}$;
\item $\tilde{\xi}:=U\xi U^{-1}\in \tilde{M}^{\tilde{M}}$ for any
$\xi\in M^{M}$; \item the map  $\tilde{\xi}:=U\xi U^{-1}$ is
continuous and quasi-affine if $\xi\in E$ is so; \item
$\tilde{E}:=\{\tilde{\xi}|\ \xi\in E\}$ is a compact sub-semigroup
of the semigroup $\tilde{M}^{\tilde{M}}$ .
\end{enumerate}
\end{lemma}
\begin{proof} Since the map $U:\mathfrak B\to \mathbb K^n$ is
linear, continuous and the set $M$ is compact and convex, then
$\tilde{M}=U(M)\subseteq \mathbb K^n$ is also compact and convex.

The second statement is evident.

To prove the third statement we notice that the map $\tilde{\xi}$
is continuous because it is a composition of three continuous
mappings. Let $\tilde{A}$ be a subset of $\tilde{M}$ and
$A:=U^{-1}(\tilde{A})\subseteq M$, then by linearity of $U$ we
obtain
\begin{equation}\label{eqE0}
U(coA)=coU(A) .\nonumber
\end{equation}
On the other hand since the map $\xi$ is quasi-affine we have
\begin{eqnarray}\label{eqE1}
& \tilde{\xi}(co\tilde{A})=U\xi U^{-1}(coU(A))=U\xi (co
A)\subseteq U(co(\xi(A))=\nonumber \\
& co U\xi U^{-1} U(A)=co \tilde{\xi}(\tilde{A}),\nonumber
\end{eqnarray}
i.e., the map $\tilde{\xi}$ is quasi-affine.

Consider the map $\Phi :E\to \tilde{E}$ defined by equality
$\Phi(\xi):=U\xi U^{-1}$ for any $\xi\in E$. It is clear that
$\Phi$ is continuous (with respect to pointwise topology on $E$),
$\Phi(E)=\tilde{E}$ and, consequently, $\tilde{E}$ is a compact
subset of $\tilde{M}^{\tilde{M}}$.

To finish the proof of Lemma we need to establish that $\tilde{E}$
is a sub-semigroup of the semigroup $\tilde{M}^{\tilde{M}}$. Let
$\tilde{\xi_{i}}\in \tilde{E}$ ($i=1,2$) we will show that
$\tilde{\xi_{1}}\tilde{\xi_{2}}\in \tilde{E}$. In fact. Since
$\tilde{\xi_{i}}\in \tilde{E}$ ($i=1,2$) the we get
\begin{equation}\label{eqE2}
\tilde{\xi_{1}}\tilde{\xi_{2}}=U\xi_{1}U^{-1}U\xi_{2}U^{-1}=U\xi_{1}\xi_2U^{-1}\in
\tilde{E} \nonumber
\end{equation}
because $\xi_{1}\xi_{2}\in E$, since $E$ is a semigroup.
\end{proof}

\begin{theorem}\label{th12.3.8*}
Let $\mathfrak B$ be a finite-dimensional Banach space over the
field $\mathbb K$ ($\mathbb K =\mathbb R$ or $\mathbb C$), $M$ be
a compact convex subset of $\mathfrak B$. Suppose that the
following conditions are fulfilled:
\begin{enumerate}
\item  $E$ is a compact sub-semigroup of the semigroup $M^{M}$;
\item every $\xi\in E$ is continuous and quasi-affine.
\end{enumerate}
Then there exists a common fixed point $\bar{x}\in M $ of $E$.
\end{theorem}
\begin{proof} Denote by $n:=dim \mathfrak B$ and $\mathbb K^n$ the
$n$-dimensional Euclidean space over the field $\mathbb K$, then
by Theorem \ref{thE1} there exists a linear isomorphism $\Phi$
between $\mathfrak B$ and $\mathbb E^n$. Let $\tilde{M}:=\Phi(M)$
and $\tilde{E}:=\{\Phi \xi \Phi^{-1}|\ \xi\in E\}$, then by Lemma
\ref{lE1} we have
\begin{enumerate}
\item $\tilde{M}$ is compact and convex subset of $\mathbb K^{n}$;
\item $\tilde{\xi}:=U\xi U^{-1}\in \tilde{M}^{\tilde{M}}$ for any
$\xi\in M^{M}$; \item the map  $\tilde{\xi}:=U\xi U^{-1}$ is
continuous and quasi-affine if $\xi\in E$ is so; \item
$\tilde{E}:=\{\tilde{\xi}|\ \xi\in E\}$ is a compact sub-semigroup
of the semigroup $\tilde{M}^{\tilde{M}}$ .
\end{enumerate}

Since $\mathbb K^{n}$ is a finite-dimensional uniformly convex
Banach space, then by Theorem \ref{th12.3.8} there exists at least
one point $\bar{v}\in \tilde{M}$ such that
$\tilde{\xi}(\bar{v})=\bar{v}$ for any $\tilde{\xi}\in\tilde{E}$.
To finish the proof of Theorem it is sufficient to note that
$\xi(\bar{x})=\bar{x}$ for any $\xi\in E$, where
$\bar{x}:=\Phi^{-1}(\bar{v})$.
\end{proof}

\section{Favard's theory}\label{Sec4}

Let $(\mathfrak B, |\cdot |)$ be a Banach space with the norm
$|\cdot|$, $\mathbb T \supseteq \mathbb S_{+}$ be a subsemigroup
of group $\mathbb S$ and $\langle \mathfrak B, \varphi, (Y,\mathbb
S, \sigma)\rangle$ (or shortly $\varphi$) be a linear cocycle over
dynamical system $(Y,\mathbb S,\sigma)$ with the fibre $\mathfrak
B$, i.e., $\varphi$ is a continuous mapping from $\mathbb T\times
\mathfrak B \times Y$ into $\mathfrak B$ satisfying the following
conditions:
\begin{enumerate}
\item $\varphi(0,u,y)=u$ for any $u\in\mathfrak B$ and $y\in Y$;
\item
$\varphi(t+\tau,u,y)=\varphi(t,\varphi(\tau,u,y),\sigma(\tau,y))$
for any $t,\tau\in\mathbb T$, $u\in \mathfrak B$ and $y\in Y$;
\item for any $(t,y)\in \mathbb T\times Y$ the mapping
$\varphi(t,\cdot,y):\mathfrak B\mapsto \mathfrak B$ is linear.
\end{enumerate}

Denote by $[\mathfrak B]$\index{$[\mathfrak B]$} the Banach space
of any linear bounded operators $A$ acting on the space $\mathfrak
B$ equipped with the operator norm
$||A||:=\sup\limits_{|x|\le1}|Ax|$\index{$||A||$}.

\begin{example}\label{exLS1}  Let $Y$ be a complete metric space and
$(Y,\mathbb R,\sigma)$ be a dynamical system on $Y$. Consider the
following linear differential equation
\begin{equation}\label{eqLS01.81}
x'=A(\sigma(t,y))x,\  \ (y\in Y)
\end{equation}
where $A\in C(Y,[\mathfrak B])$. Note that the following
conditions are fulfilled for equation (\ref{eqLS01.81}):
\begin{enumerate}
\item[a.] for any $ u \in \mathfrak B $ and $y\in Y $ equation
(\ref{eqLS01.81}) has exactly one solution that is defined on $
\mathbb R $ and satisfies the condition $ \varphi (0,u,y) = u ;$
\item[b.] the mapping $ \varphi : (t,u,y) \to \varphi (t,u,y) $ is
continuous in the topology of $ \mathbb R\times \mathfrak B \times
Y$.
\end{enumerate}

Under the above assumptions equation (\ref{eqLS01.81}) generates a
linear cocycle $\langle \mathfrak B, \varphi, (Y,\mathbb R,
\sigma)\rangle$ over dynamical system $(Y,\mathbb R,\sigma)$ with
the fibre $\mathfrak B$.
\end{example}

\begin{example}\label{exLS02} Consider differential equation
\begin{equation}\label{eqLS02..8}
x'=A(t)x,
\end{equation}
where $A\in C(\mathbb R,[\mathfrak B])$. Along this equation
(\ref{eqLS02..8}) consider its $H$-class, i.e., the following
family of equations
\begin{equation}\label{eqLS03}
x'=B(t)x,
\end{equation}
where $B\in H(A)$. Note that the following conditions are
fulfilled for equation (\ref{eqLS02..8}) and its $H$-class
(\ref{eqLS03}):
\begin{enumerate}
\item[a.] for any $ u \in \mathfrak B $ and $ B \in H(A) $
equation (\ref{eqLS03}) has exactly one solution $ \varphi
(t,u,B)$ satisfying the condition $ \varphi (0, u, B ) = v $;
\item[b.] the mapping $\varphi : (t,u,B ) \to \varphi (t,u,B )$ is
continuous in the topology of $\mathbb R \times \mathfrak B \times
C(\mathbb R; [\mathfrak B])$.
\end{enumerate}

Denote by $(H(A),\mathbb R,\sigma)$ the shift dynamical system on
$H(A)$. Under the above assumptions equation (\ref{eqLS02..8})
generates a linear cocycle $\langle \mathfrak B, \varphi,
(H(A),\mathbb R, \sigma)\rangle$ over dynamical system
$(H(A),\mathbb R,\sigma)$ with the fibre $\mathfrak B$.

Note that equation (\ref{eqLS02..8}) and its $H$-class can be
written in the form (\ref{eqLS01.81}). In fact. We put $Y:=H(A)$
and denote by $\mathcal A \in C(Y,[\mathfrak B])$ defined by
equality $\mathcal A(B):=B(0)$ for any $B\in H(A)=Y$, then
$B(\tau)=\mathcal A(\sigma(B,\tau))$ ($\sigma(\tau,B):=B_{\tau}$,
where $B_{\tau}(t):=B(t+\tau)$ for any $t\in\mathbb R$). Thus
equation (\ref{eqLS02..8}) with its $H$-class can be rewrite as
follow
\begin{equation}\label{eqLS04}
x'=\mathcal A(\sigma(t,B))x.  \ (B\in H(A))\nonumber
\end{equation}
\end{example}


\begin{definition}\label{def2.11.1}
Let $\langle (X,\mathbb{T}_{+},\pi),\, (Y,\mathbb{T},\sigma),
h\rangle$ be a linear nonautonomous (affine) dynamical system. A
nonautonomous dynamical system $\langle
(W,$$\mathbb{T}_{+},$$\mu),$$\, ({Z},$$\mathbb{T},$$\lambda),$$
\varrho\rangle$ is said to be linear non-homogeneous, generated by
linear (homogeneous) dynamical system $\langle
(X,\mathbb{T}_{+},\pi),$ $(Y,\mathbb{T},\sigma), h\rangle$, if the
following conditions hold:
 \begin{enumerate}
 \item[1.]
there exits a homomorphism $q$ of the dynamical system
$({Z},\mathbb{T},\lambda)$ onto $(Y,\mathbb{T},\sigma)$; \item[2.]
the space $W_y:=(q\circ \rho)^{-1}(y)$ is affine for all $y\in
(q\circ \varrho)(W)\subseteq Y$ and the vectorial space
$X_y=h^{-1}(y)$ is an associated space to $W_y$
(\cite[p.175]{shv1}). The mapping $\mu^t:W_y\to W_{\sigma^ty}$ is
affine and $\pi^t:X_y\to X_{\sigma^ty}$ is its linear associated
function (\cite[p.179]{shv1}), i.e., $X_y=\{w_1-w_2\ \vert \
w_1,w_2\in W_y\}$ and $\mu^tw_1-\mu^tw_2=\pi^t(w_1-w_2)$ for all
$w_1,w_2\in W_y$ and $t\in \mathbb T_{+}$.
 \end{enumerate}
\end{definition}

\begin{remark}\label{r2.11.6}
The definition of linear non-homogeneous system\index{linear
non-homogeneous system}, associated by the given linear system, is
given in the work \cite{bro84}, but our definition is more general
and sometimes more flexible.
\end{remark}

Let $\langle \mathfrak B, \varphi, (Y,\mathbb T, \sigma)\rangle$
be a linear cocycle over dynamical system $(Y,\mathbb T,\sigma)$
with the fibre $\mathfrak B$, $f\in C(Y,\mathbb B)$ and $\psi$ be
a mapping from $\mathbb T\times \mathfrak B \times Y$ into
$\mathfrak B$ defined by equality
\begin{equation}\label{eqLS5.8}
\psi(t,u,y):=U(t,y)u+\int_{0}^{t}U(t-\tau,\sigma(\tau,y))f(\sigma(\tau,y))d\tau
\ \ \mbox{if}\ \mathbb T=\mathbb R \nonumber
\end{equation}
and
\begin{equation}\label{eqLS6.8}
\psi(t,u,y):=U(t,y)u+ \sum_{\tau
=0}^{t}U(t-\tau,\sigma(\tau,y))f(\sigma(\tau,y)) \ \ \mbox{if}\
\mathbb T=\mathbb Z.\nonumber
\end{equation}

From the definition of mapping $\psi$ it follows that $\psi$
possesses the following properties:
\begin{enumerate}
\item[1.] $\psi(0,u,y)=u$ for any $(u,y)\in \mathfrak B\times Y$;
\item[2.] $\psi(t+\tau,u,y)=\psi(t,\psi(\tau,u,y),\sigma(\tau,y))$
for any $t,\tau\in \mathbb T$ and $(u,y)\in \mathfrak B\times Y$;
\item[3.] the mapping $\psi :\mathbb T\times \mathfrak B\times
Y\mapsto \mathfrak B$ is continuous; \item[4.] $\psi (t, u,y)-
\psi(t,v,y)= \varphi(t,u-v,y)$ for any $t\in\mathbb T$, $u,v\in
\mathfrak B$ and $y\in Y$, i.e., the mapping
$\psi(t,\cdot,y):\mathfrak B\mapsto \mathfrak B$ is affine for
every $(t,y)\in \mathbb T\times Y$.
\end{enumerate}

\begin{definition}\label{defAF1} A triplet $\langle \mathfrak B,\psi, (Y,\mathbb
T,\sigma)\rangle$ is called an affine (nonhomogeneous) cocycle
\index{an affine (nonhomogeneous) cocycle} over dynamical system
$(Y,\mathbb T,\sigma)$ with the fibre $\mathfrak B$, if the $\psi$
is a mapping from $\mathbb T\times \mathfrak B\times Y$ into
$\mathfrak B$ possessing the properties 1.-4.
\end{definition}

\begin{remark}\label{remNH1} If we have a linear cocycle $\langle \mathfrak B, \varphi,
(Y,\mathbb T, \sigma)\rangle$ over dynamical system $(Y,\mathbb
T,\sigma)$ with the fibre $\mathfrak B$ and $f\in C(Y,\mathbb B)$,
then by equality (\ref{eqLS5.8}) (respectively, by
(\ref{eqLS6.8})) is defined an affine cocycle $\langle \mathfrak
B, \psi, (Y,\mathbb T, \sigma)\rangle$ over dynamical system
$(Y,\mathbb T,\sigma)$ with the fibre $\mathfrak B$ which is
called an affine (nonhomogeneous) cocycle associated by linear
cocycle $\varphi$ and the function $f\in C(Y,\mathfrak B)$.
\end{remark}

\begin{lemma}\label{lB1} Let $\langle \mathfrak B,\varphi, (Y,\mathbb
T,\sigma)\rangle$ be a linear cocycle and $\mathfrak B$ be a
finite dimensional Banach space. Then there exists a positive
constant $L$ such that
\begin{equation}\label{eqB1}
|\varphi(t,u,y)|\le L|u|
\end{equation}
for any $t\ge 0$ and $u\in \mathfrak B^{+}_{y}:=\{x\in \mathfrak B
|\ \sup\limits_{t\ge 0}|\varphi(t,u,y)|<+\infty\}$.
\end{lemma}
\begin{proof}
Since the cocycle $\varphi$ is linear and $\mathfrak B$ is finite
dimensional, then $\mathfrak B^{+}_{y}$ is a subspace of the
Banach space $\mathfrak B$. Consider the family of linear bounded
operators $\mathfrak A :=\{\varphi(t,\cdot,y)|\ t\ge 0\}$. Note
that for any $u\in \mathfrak B^{+}_{y}$ there exists a positive
number $C(u)$ (for example $C(u)=\sup\limits_{t\ge
0}|\varphi(t,u,y)|$) such that
$$
|Au|\le C(u)
$$
for any $A\in \mathfrak A$. By Banach-Steinhaus theorem the family
operators $\mathfrak A$ is bounded, i.e., there exists a positive
constant $L$ such that (\ref{eqB1}) takes place. Lemma is proved.
\end{proof}

\begin{lemma}\label{l3.1}\cite{CC_2009}
Let $\langle E,\varphi, (Y,\mathbb T,\sigma)\rangle$ be a cocycle
and $ \langle (X,\mathbb T_{+},\pi),(Y,\mathbb T,\sigma),h\rangle
$ be the nonautonomous dynamical system generated by the cocycle
$\varphi$. Assume that $x_0:=(u_0,y_0)\in X=E\times Y$ and the set
$Q_{(u_0,y_0)}^{+}:=\overline{\{\varphi(t,u_0,y_0): t\in \mathbb
T_{+}\}}$ is compact. Then the semi-hull
$H^{+}(x_0):=\overline{\{\pi(t,x_0)|\ t\in\mathbb T_{+}\}}$ is
conditionally compact.
\end{lemma}

\begin{lemma}\label{lB2} Let $\langle \mathfrak B,\psi, (Y,\mathbb
T,\sigma)\rangle$ be an affine cocycle and $\langle (X,\mathbb
T_{+},\pi),(Y,\mathbb T,\sigma),h\rangle$ be a nonautonomous
dynamical system generated by cocycle $\varphi$ ($X:=\mathfrak
B\times Y, \pi :=(\varphi,\sigma)$ and $h:=pr_{2}$). Assume that
the following conditions are fulfilled:
\begin{enumerate}
\item the Banach space $\mathfrak B$ is finite dimensional; \item
the point $y_0\in Y$ is Poisson stable in the positive direction;
\item there exits a point $u_0\in \mathfrak B$ such that $\psi
(\mathbb T_{+},u_0,y_0)$ is relatively compact.
\end{enumerate}

Then the following statement hold:
\begin{enumerate}
\item the set $K:=\omega_{x_0}\subset X=\mathfrak B\times Y$ is
conditionally compact, where $x_0:=(u_0,y_0)$; \item
$M:=\overline{co}K_{y_0}$ is a compact convex subset of $X_{y_0}:=
\mathfrak B\times \{y_0\}$, where $K_{y_0}:=\omega_{x_0}\bigcap
X_{y_0}$; \item $\mathcal{E}^{+}_{y_0}$ is a compact sub-semigroup
of the semi-group $M^{M}$; \item every $\xi\in
\mathcal{E}^{+}_{y_0}$ is affine and continuous.
\end{enumerate}
\end{lemma}
\begin{proof}
The first statement follows from Lemma \ref{l3.1}.

Since the set $K$ is conditionally compact, then the set $K_{y_0}$
is compact and, consequently, the set $M=\overline{co}K_{y_0}$ is
also compact.

The third statement follows from Lemma \ref{l9.2.4}.

Let $m$ be an arbitrary natural number,
$\alpha_1,\alpha_2,\ldots,\alpha_{m}\in \mathbb R_{+}$ with
\begin{equation}\label{eqD6}
\sum_{k=1}^{m}\alpha_{k}=1 \nonumber
\end{equation}
and $x_1,x_2,\ldots,x_{m}\in M$ ($x_i=(u_i,y_0)$
$i=1,2,\ldots,m$). Since the maps from $\{\varphi(t,\cdot,y_0)|\
t\ge 0\}$  are affine, then we have
\begin{equation}\label{eqD7}
\psi(t,\sum_{k=1}^{m}\alpha_{k}u_k)=\sum_{k=1}^{m}\alpha_{k}\psi(t,u_k,y_0).
\end{equation}
for any $t\ge 0$. Let now $\xi\in \mathcal{E}^{+}_{y_0}$, then
there exists a sequence $\{t_{n}\}\in \mathfrak N_{y_0}^{+\infty}$
such that
\begin{equation}\label{eqD7.1}
\lim\limits_{n\to \infty} \pi^{t_n}(x)=\xi(x)
\end{equation}
for any $x\in M$. From (\ref{eqD7}) we get
\begin{equation}\label{eqD7.i}
\psi(t_n,\sum_{k=1}^{m}\alpha_{k}u_k,y_0)=\sum_{k=1}^{m}\alpha_{k}\psi(t_n,u_k,y_0)
\end{equation}
for any $n\in \mathbb N$. Passing to the limit in (\ref{eqD7.i})
and taking in consideration (\ref{eqD7.1}) we obtain
\begin{equation}\label{eqD8}
\xi(\sum_{k=1}^{m}\alpha_{k}x_k)=\sum_{k=1}^{m}\alpha_{k}\xi(x_k)
\nonumber.
\end{equation}
Thus the map $\xi$ is affine.

Let $x\in M$ and $\xi \in \mathcal{E}^{+}_{y_0}$, then $x=(u,y_0)$
and there exists a sequence $\{t_n\}\in \mathfrak
N_{y_0}^{+\infty}$ such that $\xi(x)=\lim\limits_{n\to
\infty}(\psi(t_n,u,y_0),\sigma(t_n,y_0))=(\nu(u),y_0)$, where
\begin{equation}\label{eqC0}
\nu(u)=\lim\limits_{n\to \infty}\psi(t_n,u,y_0)
\end{equation}
for any $(u,y_0)\in M$. By Lemma \ref{lB1} there exists a positive
constant $L$ such that
\begin{equation}\label{eqC0.1}
|\psi(t_n,u_1,y_0)-\psi(t_n,u_2,y_0)|\le L|u_1-u_2|
\end{equation}
for any $(u_1,y_0),(u_2,y_0)\in M$ and $n\in \mathbb N$. Passing
to the limit in (\ref{eqC0.1}) as $n\to \infty$ and taking in
consideration (\ref{eqC0}) we obtain
\begin{equation}\label{eqC1}
\rho(\xi(x_1),\xi(x_2))\le L\rho(x_1,x_2)\nonumber
\end{equation}
for every map $\xi\in \mathcal{E}^{+}_{y_0}$ because
$\xi(x)=(\nu(u),y_0)$ for any $x=(u,y_0)\in M$ and
$\rho(x_1,x_2)=|u_1-u_2|$ ($x_i=(u_i,y_0), i=1,2$). Lemma is
completely proved.
\end{proof}

\begin{theorem}\label{thB2} Let $\langle \mathfrak B,\psi, (Y,\mathbb
T,\sigma)\rangle$ be an affine cocycle and $\langle (X,\mathbb
T_{+},\pi),(Y,\mathbb T,\sigma),h\rangle$ be a nonautonomous
dynamical system generated by cocycle $\varphi$ ($X:=\mathfrak
B\times Y, \pi :=(\varphi,\sigma)$ and $h:=pr_{2}$). Assume that
the following conditions are fulfilled:
\begin{enumerate}
\item the Banach space $\mathfrak B$ is finite dimensional; \item
the point $y_0\in Y$ is Poisson stable; \item there exits a point
$u_0\in \mathfrak B$ such that $\psi (\mathbb T_{+},u_0,y_0)$ is
relatively compact.
\end{enumerate}

Then there exists at least one point $\bar{u}\in \mathfrak B$ such
that $\mathfrak N_{y_0}^{+\infty}\subseteq \mathfrak
N_{\bar{x}}^{+\infty}$, i.e., the point $\bar{x}$ is comparable
with the point $y_0$.
\end{theorem}
\begin{proof}
Let $x_0:=(u_0,y_0)$ and $K:=\omega_{x_0}$, then by Lemma
\ref{lB2} we have
\begin{enumerate}
\item the set $K:=\omega_{x_0}\subset X=\mathfrak B\times Y$ is
conditionally compact, where $x_0:=(u_0,y_0)$; \item
$M:=\overline{co}K_{y_0}$ is a compact convex subset of $X_{y_0}:=
\mathfrak B\times \{y_0\}$, where $K_{y_0}:=\omega_{x_0}\bigcap
X_{y_0}$; \item $\mathcal{E}^{+}_{y_0}$ is a compact sub-semigroup
of the semi-group $M^{M}$; \item every $\xi\in
\mathcal{E}^{+}_{y_0}$ is affine and continuous.
\end{enumerate}
According to Theorem \ref{th12.3.8*} there exists at least one
point $\bar{x}=(\bar{u},y_0)\in M$ such that
$\xi(\bar{x})=\bar{x}$ for any $\xi\in \mathcal{E}^{+}_{y_0}$. Now
we show that the point $\bar{x}$ is comparable by character of
recurrence with the point $y_0$, i.e., $\mathfrak
N_{y_0}^{+\infty}\subseteq \mathfrak N_{\bar{x}}^{+\infty}$. In
fact. Let $\{t_n\}\in \mathfrak N_{y_0}^{+\infty}$, then
$\sigma(t_n,y_0)\to y_0$ as $n\to \infty$. Since
$\Sigma_{\bar{x}}$ is conditionally precompact and
$\{\pi(t_n,\bar{x})\}=\Sigma_{\bar{x}}\bigcap
h^{-1}(\{\sigma(t_n,y_0)\})$, then $\{\pi(t_n,\bar{x})\}$ is a
precompact sequence. To show that $\{t_n\}\in\mathfrak
N_{\bar{x}}^{+\infty}$ it is sufficient to prove that the sequence
$\{\pi(t_n,\bar{x})\}$ has at most one limiting point. Let $p_{i}$
($i=1,2$) be two limiting points of $\{\pi(t_n,\bar{x})\}$, then
there are $\{t_{k^{i}_{n}}\}\subseteq \{t_n\}$ such that
$p_{i}:=\lim\limits_{n\to \infty}\pi(t_{k^{i}_{n}},x_0)$
($i=1,2$). Notice that the set
$$
Q:= \overline {\bigcup \{\pi ^{t_{n}}(X_{y})| n \in \mathbb N \}}
$$
is compact, because $X$ is conditionally compact. Thus $\{\pi
^{t_{n}}|_{X_{y}}\} \subseteq Q^{X_{y}}$ and according to
Tykhonov's theorem this sequence is relatively compact and,
consequently, without loss of generality we can suppose that the
subsequences $\{\pi^{t_{k^{i}_{n}}}\}\subset \{\pi^{t_n}\}$
($i=1,2$) are convergent. Denote by $\xi^{i}=\lim\limits_{n\to
\infty}\pi^{t_{k^{i}_{n}}}$, then $\xi^{i}\in
\mathcal{E}^{+}_{y_0}$ ($i=1,2$) and
$p^{i}=\xi^{i}(\bar{x})=\bar{x}$. Thus we have
$p^{1}=\bar{x}=p^{2}$. Theorem is completely proved.
\end{proof}

\begin{coro}\label{corB1} Let $\langle \mathfrak B,\psi, (Y,\mathbb
T,\sigma)\rangle$ be an affine cocycle. Under the conditions of
Theorem \ref{thB2} if the point $y_0\in Y$ is $\tau$-periodic
(respectively, Bohr almost periodic, almost automorphic, recurrent
in the sense of Birkhoff, Levitan almost periodic, almost
recurrent in the sense of Bebutov, Poisson stable), then there
exists at least one point $\bar{x}\in X:=\mathfrak B\times Y$ such
that $\bar{x}$ has the same character of recurrence as $y_0$,
i.e., $\bar{x}$ is $\tau$-periodic (respectively, almost
automorphic, recurrent in the sense of Birkhoff, Levitan almost
periodic, almost recurrent in the sense of Bebutov, Poisson
stable).
\end{coro}
\begin{proof} This statement follows from Theorems \ref{thB2} and
\ref{thPC4}.
\end{proof}

\section{Applications}\label{Sec5}

\subsection{Linear Differential Equations}\label{Ssec5.1}

\begin{example}\label{exNH1}  Let $Y$ be a complete metric space,
$(Y,\mathbb R,\sigma)$ be a dynamical system on $Y$ and $
[\mathfrak B]$ be the space of linear bounded operators acting
into Banach space $ \mathfrak B $ equipped with the operator norm
and $f\in C(Y,\mathfrak B)$. Consider the following linear
nonhomogeneous differential equation
\begin{equation}\label{eqNH1}
x'=A(\sigma(t,y))x +f(\sigma(t,y)),\  \ (y\in Y)
\end{equation}
where $A\in C(Y,[\mathfrak B])$.

Equation (\ref{eqLS01.81}) generates a linear cocycle $\langle
\mathfrak B, \varphi, (Y,\mathbb R, \sigma)\rangle$ over dynamical
system $(Y,\mathbb R,\sigma)$ with the fibre $\mathfrak B$.
According to Remark \ref{remNH1} by equality (\ref{eqLS5.8}) is
defined a linear nonhomogeneous cocycle $\langle B,\psi,(Y,\mathbb
R,\sigma)\rangle$ over dynamical system $(Y,\mathbb R,\sigma)$
with the fibre $\mathfrak B$. Thus every nonhomogeneous linear
differential equations (\ref{eqNH1}) generates a linear
nonhomogeneous cocycle $\psi$.
\end{example}

\begin{example}\label{exLS02.8}{\rm
Consider a linear nonhomogeneous differential equation
\begin{equation}\label{eqNH2}
x'=A(t)x +f(t),
\end{equation}
where $f\in C(\mathbb R,\mathfrak B)$ and $A\in C(\mathbb
R,[\mathfrak B])$. Along this equation (\ref{eqNH2}) consider its
$H$-class, i.e., the following family of equations
\begin{equation}\label{eqNH3}
x'=B(t)x +g(t),
\end{equation}
where $(B,g)\in H(A,f)$. Notice that the following conditions are
fulfilled for equation (\ref{eqLS02..8}) and its $H$-class
(\ref{eqLS03}):
\begin{enumerate}
\item[a.] for any $ u \in \mathfrak B $ and $ B \in H(A) $
equation (\ref{eqLS03}) has exactly one solution $ \varphi
(t,u,B)$ defined on $\mathbb R$ and the condition $ \varphi (0, u,
B ) = v$ is fulfilled; \item[b.] the mapping $ \varphi : (t,u,B )
\to \varphi (t,u,B ) $ is continuous in the topology of $\mathbb R
\times \mathfrak B \times C(\mathbb R ; \Lambda )$.
\end{enumerate}

Denote by $(H(A,f),\mathbb R,\sigma)$ the shift dynamical system
on $H(A,f)$. Under the above assumptions the equation
(\ref{eqLS02..8}) generates a linear cocycle $\langle \mathfrak B,
\varphi, (H(A,f),\mathbb R, \sigma)\rangle$ over dynamical system
$(H(A,f),\mathbb R,\sigma)$ with the fibre $\mathfrak B$. Denote
by $\psi$ a mapping from $\mathbb R_{+}\times \mathfrak B\times
H(A,f)$ into $\mathfrak B$ defined by equality
\begin{equation}\label{eqNH4}
\psi(t,u,(B,g)):=U(t,B)u+\int_{0}^{t}U(t-\tau,B_{\tau})g(\tau)d\tau,\nonumber
\end{equation}
then $\psi$ possesses the following properties:
\begin{enumerate}
\item[(i)] $\psi(0,u,(B,g))=u$ for any $(u,(B,g))\in \mathfrak
B\times H(A,f)$; \item[(ii)]
$\psi(t+\tau,u,(B,g))=\psi(t,\psi(\tau,u,(B,g)),(B_{\tau},g_{\tau}))$
for any $t,\tau\in \mathbb R$ and $(u,(B,g))\in \mathfrak B\times
H(A,f)$; \item[(iii)] the mapping $\psi :\mathbb R\times \mathfrak
B\times H(A,f)\mapsto \mathfrak B$ is continuous; \item[(iv)]
$\psi (t,\lambda u+\mu v,(B,g))=\lambda \psi(t,u,(B,g))+\mu
\psi(t,v,(B,g))$ for any $t\in\mathbb R$, $u,v\in \mathfrak B$,
$(B,g)\in H(A,f)$ and $\lambda,\mu\in \mathbb R$ (or $\mathbb C$)
with the condition $\lambda +\mu =1$, i.e., the mapping
$\psi(t,\cdot,(B,g)):\mathfrak B\mapsto \mathfrak B$ is affine for
every $(t,(B,g))\in \mathbb R\times H(A,f)$.
\end{enumerate}

Thus, every linear nonhomogeneous differential equation of the
form (\ref{eqNH2}) (and its $H$-class (\ref{eqNH3})) generates a
linear nonhomogeneous cocycle $\langle \mathfrak B, \psi,
(H(A,f),\mathbb R,\sigma)\rangle$ over dynamical system
$(H(A,f),\mathbb R,\sigma)$ with the fibre $\mathfrak B$.}
\end{example}

\begin{theorem}\label{thB2.1} Assume that the
following conditions are fulfilled:
\begin{enumerate}
\item $\mathfrak B$ is a finite dimensional Banach space; \item
the function $f\in C(\mathbb R,\mathfrak B)$ and matrix-function
$A\in C(\mathbb R,[\mathfrak B])$ are jointly Poisson stable;
\item there exits a bounded on $\mathbb R_{+}$ solution
$\psi(t,u_0,A,f)$ of equation (\ref{eqNH2}).
\end{enumerate}

Then there exists at least one compatible solution
$\psi(t,\bar{u},A,f)$ of equation (\ref{eqNH2}), i.e., $\mathfrak
N_{(A,f)}^{+\infty}\subseteq \mathfrak N_{\bar{\psi}}^{+\infty}$,
where $\bar{\psi}:=\psi(\cdot,\bar{u},A,f)$.
\end{theorem}
\begin{proof} Let $\langle \mathfrak B, \psi, (H(A,f),\mathbb R,\sigma)
\rangle$ be a cocycle generated by equation (\ref{eqNH2}). By
Example \ref{exLS02.8} the cocycle $\psi$ is affine. Now applying
Theorem \ref{thB2} to constructed cocycle $\psi$ we complete the
proof of Theorem.
\end{proof}

\begin{coro}\label{corB2.1*} Under the conditions of Theorem
\ref{thB2.1} if the function $(A,f)\in C(\mathbb R,[\mathfrak
B])\times C(\mathbb R,\mathfrak B)$ is $\tau$-periodic
(respectively, Bohr almost periodic, almost automorphic, Levitan
almost periodic, almost recurrent in the sense of Bebutov,
recurrent in the sense of Birkhoff, Poisson stable), then equation
(\ref{eqNH2}) has at least one $\tau$-periodic (respectively,
almost automorphic, Levitan almost periodic, almost recurrent in
the sense of Bebutov, recurrent in the sense of Birkhoff, Poisson
stable) solution.
\end{coro}
\begin{proof} If the function $(A,f)\in C(\mathbb R,[\mathfrak
B])\times C(\mathbb R,\mathfrak B)$ is $\tau$-periodic
(respectively, Levitan almost periodic, almost recurrent in the
sense of Bebutov, Poisson stable), then this statement follows
from Theorem \ref{thB2.1}.

Suppose that $(A,f)\in C(\mathbb R,[\mathfrak B])\times C(\mathbb
R,\mathfrak B)$ is Bohr almost periodic (respectively, almost
automorphic, recurrent in the sense of Birkhoff), then it is
Levitan almost periodic. By above arguments equation (\ref{eqNH2})
has at least one Levitan almost periodic solution $\bar{\psi}$. On
the other hand the solution $\bar{\psi}$ is bounded and uniformly
continuous on $\mathbb R$ because $(A,f)\in C(\mathbb R,[\mathfrak
B])\times C(\mathbb R,\mathfrak B)$ is bounded on $\mathbb R$.
Thus the function $\bar{\psi}$ is Levitan almost periodic and
stable in the sense of Lagrange and, consequently, it is almost
automorphic (respectively, recurrent in the sense of Birkhoff).
\end{proof}

\begin{remark}\label{remB1} 1. If the function $(A,f)\in C(\mathbb R,[\mathfrak
B])\times C(\mathbb R,\mathfrak B)$ is $\tau$-periodic, then
Corollary \ref{corB2.1} coincides with the Massera's theorem.

2. If the function $(A,f)\in C(\mathbb R,[\mathfrak B])\times
C(\mathbb R,\mathfrak B)$ is Bohr almost periodic, then Corollary
\ref{corB2.1} improves Theorem of Levitan and also Zhikov's
theorems.
\end{remark}

\subsection{Linear Difference Equations}\label{Ssec5.2}

\begin{example}\label{ex14.1}
Consider difference equation
\begin{equation}\label{eq113.2.1}
u(t+1) = A(t)u(t) +f(t),
\end{equation}
where $(A,f)\in C(\mathbb Z,[\mathfrak B])\times C(\mathbb Z,
\mathfrak B).$ Along with equations (\ref{eq113.2.1}) we consider
also its $H$-class (\ref{eq113.2.1}), that is the family of
equations
\begin{equation}\label{eq113.2.2}
v(t+1) = B(t)v(t) +g(t),
\end{equation}
with $(B,g)\in H(A,f):= \overline{\{(A_{\tau},f_{\tau})\ |\ \tau
\in \mathbb Z\}}$, where the bar denotes closure in $C(\mathbb Z,
[ E])\times C(\mathbb Z, E)$. Let $\varphi (t, v, (B,g))$ be the
solution of equation (\ref{eq113.2.2}) that satisfies the
condition $\varphi (0, v,$ $(B,g))$ $ = $ $v$.

We put $Y:= H(A,f)$ and denote the shift dynamical system on
$H(A,f)$ by $(Y,\mathbb Z,\sigma ).$ We put $X := E\times Y$ and
define a dynamical system on $X$ by setting $\pi(\tau,(v,B,g)) :=
(\varphi(\tau, v, (B,g)), B_{\tau},g_{\tau})$ for any $(v, (B,g))
\in
 E\times Y$ and $\tau\in \mathbb Z_{+}.$ Then $\langle
(X,\mathbb Z_{+},\pi),$ $(Y,\mathbb Z,\sigma ),$ $h\rangle$ is a
semigroup nonautonomous dynamical system, where $h := pr_2 : X \to
Y.$
\end{example}

A solution $\varphi$ of equation (\ref{eq113.2.1}) defined on
$\mathbb Z$ is called \cite{scher85} compatible by the character
of recurrence if $\mathfrak N_{(A,f)}\subseteq \mathfrak
N_{\varphi},$ where $\mathfrak N_{(A,f)}:=\{\{t_n\}\subset \mathbb
Z\ |\ (A_{t_n},f_{t_n})\to (A,f)\}$ (respectively, $\mathfrak
N_{\varphi}:=\{\{t_n\}\subset \mathbb Z\ |\ \varphi_{t_n}\to
\varphi \}$).

\begin{theorem}\label{thDB2.1} Assume that the
following conditions are fulfilled:
\begin{enumerate}
\item $\mathfrak B$ is a finite dimensional Banach space; \item
the function $f\in C(\mathbb Z,\mathfrak B)$ and matrix-function
$A\in C(\mathbb Z,[\mathfrak B])$ are jointly Poisson stable;
\item there exits a bounded on $\mathbb Z_{+}$ solution
$\psi(t,u_0,A,f)$ of equation (\ref{eq113.2.1}).
\end{enumerate}

Then the exist at least one compatible solution
$\psi(t,\bar{u},A,f)$ of equation (\ref{eq113.2.1}), where
$\bar{\psi}:=\psi(\cdot,\bar{u},A,f)$.
\end{theorem}
\begin{proof} Let $\langle \mathfrak B, \psi, (H(A,f),\mathbb R,\sigma)
\rangle$ be a cocycle generated by equation (\ref{eq113.2.1}). By
Example \ref{ex14.1} the cocycle $\psi$ is affine. Now applying
Theorem \ref{thB2} to constructed cocycle $\psi$ we complete the
proof of Theorem.
\end{proof}

\begin{coro}\label{corB2.1} Under the conditions of Theorem
\ref{thDB2.1} if the function $(A,f)\in C(\mathbb Z,[\mathfrak
B])\times C(\mathbb Z,\mathfrak B)$ is $\tau$-periodic
(respectively, Bohr almost periodic, almost automorphic, Levitan
almost periodic, almost recurrent in the sense of Bebutov,
recurrent in the sense of Birkhoff, Poisson stable), then equation
(\ref{eq113.2.1}) has at least one $\tau$-periodic (respectively,
almost automorphic, Levitan almost periodic, almost recurrent in
the sense of Bebutov, recurrent in the sense of Birkhoff, Poisson
stable) solution.
\end{coro}
\begin{proof} If the function $(A,f)\in C(\mathbb Z,[\mathfrak
B])\times C(\mathbb Z,\mathfrak B)$ is $\tau$-periodic
(respectively, Levitan almost periodic, almost recurrent in the
sense of Bebutov, Poisson stable), then this statement follows
from Theorem \ref{thDB2.1}.

Suppose that $(A,f)\in C(\mathbb Z,[\mathfrak B])\times C(\mathbb
Z,\mathfrak B)$ is Bohr almost periodic (respectively, almost
automorphic, recurrent in the sense of Birkhoff), then it is
Levitan almost periodic. By above arguments equation
(\ref{eq113.2.1}) has at least one Levitan almost periodic
solution $\bar{\psi}$ (respectively, almost recurrent in the sense
of Bebutov). On the other hand the solution $\bar{\psi}$ is
bounded on $\mathbb Z$. Thus the function $\bar{\psi}$ is Levitan
almost periodic and stable in the sense of Lagrange and,
consequently, it is almost automorphic (respectively, recurrent).
\end{proof}

\subsection{Linear Functional-Difference Equations with Finite
Delay}\label{Ssec5.3}

Let $r\in\mathbb{Z}_{+}$. Let $[-r,0]$ (respectively, $[a,b]$ with
$a,b\in \mathbb Z$) be an interval in $\mathbb{Z}$ and $\mathfrak
B$ be a finite-dimensional Banach space with the norm $|\cdot|$.
We denote by $\mathcal{C}:=C([-r,0],\mathfrak B)$ (respectively,
$C([a,b],\mathfrak B)$) the Banach space of all functions
$\varphi:[-r,0] \rightarrow E$ with the norm
\begin{equation}\label{eqD1}
||\varphi|| :=\sum_{k=0}^{r} |\varphi(-k)| \nonumber
\end{equation}
for any $\varphi\in \mathcal{C}$. Note that the Banach space
$\mathcal{C}$ is finite dimensional because it is isomorphic to
the space $\mathfrak B^{r+1}$.

Let $c,a\in\mathbb{Z},\ a\geq 0,$ and $u\in C([c-r,c+a],\mathfrak
B).$ We define $u_{t}\in \mathcal{C}$ for any $t\in\lbrack c,c+a]$
by the relation $u_{t}(s):=u(t+s),-r\leq s\leq0.$ Let
$\mathfrak{A}=\mathfrak{A}(\mathcal{C},\mathfrak B)$ be the Banach
space of all linear operators that act from $\mathcal{C}$ into
$\mathfrak B$ equipped with the operator norm, let
$C(\mathbb{Z},\mathfrak{A})$ be the space of all operator-valued
functions $A:\mathbb{Z}\rightarrow\mathfrak{A}$ with the
compact-open topology, and let $(C(\mathbb{Z},\mathfrak{A}),\mathbb{Z}%
,\sigma)$ be the group dynamical system of shifts on $C(\mathbb{Z}%
,\mathfrak{A}).$ Let $H(A):=\overline{\{A_{\tau}\ |\
\tau\in\mathbb{Z}\}},$ where $A_{\tau}$ is the shift of the
operator-valued function $A$ by $\tau$ and the bar denotes closure
in $C(\mathbb{Z},\mathfrak{A}).$

\begin{example}
\label{ex04.2} Consider the non-homogeneous linear
functional-difference equation with finite delay (see, for
example, \cite{Muk,Son1})
\begin{equation}
\label{eq013.2.11}u(t+1) = A(t)u_{t} +f(t),
\end{equation}
where $A\in C(\mathbb{Z},\mathfrak{A})$ and $f\in C(\mathbb{Z},
E).$

\begin{remark}
\label{r6.1*} Denote by $\psi(t,u,(A,f))$ the solution of equation
(\ref{eq013.2.11}) defined on $\mathbb{Z}_{+}$ with initial
condition $\psi(0,u,A,f)=u\in\mathcal{C}$. By $\tilde{\psi
}(t,u,(A,f))$ we will denote below the trajectory of equation (\ref{eq013.2.11}%
), corresponding to the solution $\psi(t,u,(A,f))$, i.e., the
mapping from $\mathbb{Z}_{+}$ into $\mathcal{C}$, defined by
equality $\tilde{\psi}(t,u,(A,f))(s):=\psi(t+s,u,(A,f))$ for all
$t\in\mathbb{Z}_{+}$ and $s\in\lbrack-r,0]$.
\end{remark}

Along with equation (\ref{eq013.2.11}) we consider the family of
equations
\begin{equation}
\label{eq013.2.12}v(t+1) = B(t)v_{t} +g(t),
\end{equation}
where $(B,g)\in H(A,f):= \overline{\{(B_{\tau},f_{\tau})\ |\
\tau\in \mathbb{Z}\}}.$ Let $\tilde{\psi}(t, v,(B,g))$ be the
solution of equation (\ref{eq013.2.12}) satisfying the condition
$\tilde{\psi} (0, v, (B,g)) = v$ and defined for all $t\ge0.$ From
the general properties of linear nonhomogeneous difference
equations  we have
\begin{enumerate}
\item the mapping $\tilde{\psi}: \mathbb Z_{+}\times
\mathcal{C}\times H(A,f)\to \mathcal{C}$ is continuous; \item
$\tilde{\psi}(t+\tau,v,(B,g))=\tilde{\psi}(t,\tilde{\psi}(\tau,v,(B,g)),(B_{\tau},g_{\tau}))$
for any $t,\tau\in \mathbb Z_{+}$, $v\in \mathcal{C}$ and
$(B,g)\in H(A,f)$; \item $\tilde{\psi}(t,\alpha v_1+\beta
v_2,(B,g))=\alpha \tilde{\psi}(t,v_1,(B,g))+\beta
\tilde{\psi}(t,v_2,(B,g))$ for any $t\in\mathbb Z_{+}$,
$v_1,v_2\in \mathcal{C}$, $(B,g)\in H(A,f)$ and
$\alpha,\beta\in\mathbb R$ with $\alpha +\beta =1$.
\end{enumerate}
This means that equation (\ref{eq013.2.11}) generates a linear
nonhomogeneous cocycle $\langle \mathcal{C},$ $\tilde{\psi},$ $
(H(A,f),\mathbb Z,\sigma))\rangle$.

Let $Y := H(A,f)$ and denote the group dynamical system of shifts
on $H(A,f)$ by $(Y,\mathbb{Z},\sigma).$ Let $X :=
\mathcal{C}\times Y$ and let $\pi:= (\varphi,\sigma)$ be the
dynamical system on $X$ defined by the equality $\pi(\tau,(v,
(B,g))) := (\tilde{\psi} (\tau, v, (B,g)),B_{\tau
},g_{\tau}).$ The semi-group nonautonomous system $\langle(X,\mathbb{Z}%
_{+},\pi),$ $(Y,$ $\mathbb{Z},\sigma),$ $h\rangle$ $(h := pr_{2} :
X\to Y )$ is generated by equation (\ref{eq013.2.11}).
\end{example}

\begin{lemma}
\label{l04.3} Let $\psi(t,u,(A,f))$ be a solution of equation
(\ref{eq013.2.11}) which is relatively compact on
$\mathbb{Z}_{+},$ and let
$\langle(X,\mathbb{Z}_{+},\pi),(Y,\mathbb{Z},\sigma),h\rangle$ be
a nonautonomous dynamical system generated by equation
(\ref{eq013.2.11}). Then the set
\[
H^{+}(u,(A,f)):=\overline{\{(\tilde{\psi}(\tau,u,(A,f)),(A_{\tau},f_{\tau
}))\ |\ \tau\geq0\}}%
\]
is conditionally compact with respect to $(X,h,Y)$.
\end{lemma}

\begin{proof}
This statement follows from Lemma \ref{l3.1}.
\end{proof}

\begin{theorem}\label{thDB2.2} Assume that the
following conditions are fulfilled:
\begin{enumerate}
\item $\mathfrak B$ is a finite dimensional Banach space; \item
the function $f\in C(\mathbb Z,\mathfrak B)$ and matrix-function
$A\in C(\mathbb Z,\mathfrak{A})$ are jointly Poisson stable; \item
there exits a bounded on $\mathbb Z_{+}$ solution
$\psi(t,u_0,A,f)$ of equation (\ref{eq013.2.11}).
\end{enumerate}

Then the exist at least one compatible solution
$\psi(t,\bar{u},A,f)$ of equation (\ref{eq013.2.11}), where
$\bar{\psi}:=\psi(\cdot,\bar{u},A,f)$.
\end{theorem}
\begin{proof} Let $\langle \mathcal{C}, \tilde{\psi}, (H(A,f),\mathbb Z,\sigma)
\rangle$ be a cocycle generated by equation (\ref{eq013.2.11}). By
Example \ref{ex04.2} the cocycle $\tilde{\psi}$ is affine. Note
that the Hilbert space $\mathcal{C}$ is finite dimensional and,
consequently, it is uniformly convex. Now applying Theorem
\ref{thB2} to constructed cocycle $\tilde{\psi}$ we complete its
proof.
\end{proof}

\begin{coro}\label{corB2.2} Under the conditions of Theorem
\ref{thDB2.2} if the function $(A,f)\in C(\mathbb
Z,\mathfrak{A})\times C(\mathbb Z,\mathfrak B)$ is $\tau$-periodic
(respectively, Bohr almost periodic, almost automorphic, Levitan
almost periodic, almost recurrent in the sense of Bebutov,
recurrent in the sense of Birkhoff, Poisson stable), then equation
(\ref{eq013.2.11}) has at least one $\tau$-periodic (respectively,
almost automorphic, Levitan almost periodic, almost recurrent in
the sense of Bebutov, recurrent in the sense of Birkhoff, Poisson
stable) solution.
\end{coro}
\begin{proof} If the function $(A,f)\in C(\mathbb Z,\mathfrak{A})\times C(\mathbb Z,\mathfrak B)$ is $\tau$-periodic
(respectively, Levitan almost periodic, almost recurrent in the
sense of Bebutov, Poisson stable), then this statement follows
from Theorem \ref{thDB2.2}.

Suppose that $(A,f)\in C(\mathbb Z,\mathfrak{A})\times C(\mathbb
Z,E)$ is Bohr almost periodic (respectively, almost automorphic,
recurrent in the sense of Birkhoff), then it is Levitan almost
periodic. By above arguments equation (\ref{eq013.2.11}) has at
least one Levitan almost periodic solution $\bar{\psi}$
(respectively, almost recurrent in the sense of Bebutov). On the
other hand the solution $\bar{\psi}$ is bounded on $\mathbb Z$.
Thus the function $\bar{\psi}$ is Levitan almost periodic and
stable in the sense of Lagrange and, consequently, it is almost
automorphic (respectively, recurrent in the sense of Birkhoff).
\end{proof}

\end{document}